\documentclass[a4paper,xcolor]{article}
\usepackage{url,lineno}
\usepackage{amsfonts}
\usepackage{amssymb,amsthm,amsmath,graphicx,bm,ebezier,mathtools}
\usepackage{hyperref}
\usepackage{color,enumerate}
\usepackage{mathrsfs}
\usepackage{url,lineno}
\usepackage{arydshln}
\usepackage{lscape}
\usepackage{listings}
\usepackage{algpseudocode}
\usepackage{soul} 
\usepackage{tikz}
\usepackage[all]{xy}

\usepackage[ruled,vlined,linesnumbered]{algorithm2e}

\usepackage{multirow}
\usepackage{float}
\usepackage{fancyhdr}
\usepackage{emptypage}
\usepackage{circledsteps}

\newcommand{\NN}{{\rm N}}

\newtheorem{theorem}{Theorem}

\newtheorem{proposition}[theorem]{Proposition}
\newtheorem{corollary}[theorem]{Corollary}
\newtheorem{lemma}[theorem]{Lemma}

\theoremstyle{remark}

\newtheorem{example}[theorem]{Example}
\newtheorem{remark}[theorem]{Remark}

\def\d{\mathrm{d}}
\def\e{\mathrm{e}}
\def\g{\mathrm{g}}

\def\N{\mathbb{N}}
\def\Z{\mathbb{Z}}

\def\Q{\mathbb{Q}}

\def\G{\mathbf{G}}

\def\I{\mathbf{I}}

\def\msg{{\mathrm{ msg }}}
\def\CJ{\mathcal{J}}
\def\CL{\mathcal{L}}
\def\scrF{{\mathscr{F}}}
\def\scrP{{\mathscr{P}}}
\def\C{{\mathrm C}}
\def\ii{{\mathrm i}}
\def\m{{\mathrm m}}

\def\t{{\mathrm t}}
\def\F{{\mathrm F}}
\def\PF{{\mathrm{PF}}}
\def\Ap{\mathrm{Ap}}

\title{The complexity of a numerical semigroup}

\author{
J.\ I.\ Garc\'{\i}a-Garc\'{\i}a\footnote{
Departamento de Matem\'aticas/INDESS (Instituto Universitario para el Desarrollo Social Sostenible),
Universidad de C\'adiz, E-11510 Puerto Real (C\'{a}diz, Spain).
E-mail: ignacio.garcia@uca.es.}\\
M. A. Moreno-Frías\footnote{
Departamento de Matem\'aticas,
Universidad de C\'adiz, E-11510 Puerto Real (C\'{a}diz, Spain).
Partially supported by MTM2017-84890-P and by Junta de Andaluc\'{\i}a group FQM-298.
E-mail: mariangeles.moreno@uca.es.}
\\
J. C. Rosales \footnote{
    Dpto. de \'Algebra, Facultad de Ciencias, Universidad de Granada,
    E-18071, Granada. (Spain).
    Partially supported by MTM2017-84890-P and by Junta de Andaluc\'{\i}a group FQM-343.
    E-mail: jrosales@ugr.es.}
\\
A.\ Vigneron-Tenorio\footnote{
Departamento de Matem\'aticas/INDESS (Instituto Universitario para el Desarrollo Social Sostenible), Universidad de C\'adiz,
E-11406 Jerez de la Frontera (C\'{a}diz, Spain).
E-mail: alberto.vigneron@uca.es.}
}

\date{}

\begin{document}

\maketitle

\begin{abstract}
Let $S$ and $\Delta$ be numerical semigroups. A numerical semigroup $S$ is an $\I(\Delta)$-{\it semigroup} if $S\backslash \{0\}$ is an ideal of $\Delta$.
We will denote by $\CJ(\Delta)=\{S \mid  S \text{ is an $\I(\Delta)$-semigroup} \}.$
We will say that $\Delta$ is {\it an ideal extension of } $S$ if $S\in \CJ(\Delta).$
In this work, we present an algorithm that allows to build all the ideal extensions of a numerical semigroup. We can recursively denote by
$\CJ^0(\N)=\N,$ $\CJ^1(\N)=\CJ(\N)$ and $\CJ^{k+1}(\N)=\CJ(\CJ^{k}(\N))$ for all $k\in \N.$ The complexity of a numerical semigroup $S$ is the minimun of the set $\{k\in \N\mid S \in \CJ^k(\N)\}.$  In addition, we will give an algorithm  that allows us to compute all the numerical semigroups with  fixed multiplicity and complexity.
\end{abstract}

{\small


}

\section*{Introduction}

\hspace{0.42cm}Let $\Z=\{0,1,-1,2,-2,\dots\}$ be the set of integer numbers and
$\N=\{x\in \Z\mid x\ge 0\}$. A {\it numerical semigroup} is a
subset $S$ of $\N$ containing the zero element, closed by the sum and such that    $\N\backslash
S=\{x\in \N \mid x \notin S\}$ is finite.

If $A$ is a subset nonempty of $\N$, we denote by $\langle A
\rangle$ the submonoid of $(\N,+)$ generated by $A$, the set
$\langle A \rangle=\{\lambda_1a_1+\dots+\lambda_na_n \mid n\in
\N\setminus \{0\}, \, \{a_1,\dots, a_n\}\subseteq A,\,
\{\lambda_1,\dots,\lambda_n\}\subseteq \N\}.$ By \cite[Lemma 2.1]{libro}, we know that $ \langle A \rangle$
is a numerical semigroup if and only if $\gcd(A)=1.$
If $S$ is a numerical semigroup  and $S=\langle A \rangle$, we
say that $A$ is a {\it system of generators} of $S$. Moreover, if $S\neq
\langle B \rangle$ for all $B \varsubsetneq A$, then $A$ is called a {\it minimal system of generators} of $S$. In
\cite[Corollary 2.8]{libro} is shown that every numerical semigroup  has a unique minimal system of generator which is always finite. We denote by $\msg(S)$ the minimal system of
generators of $S$. The cardinality of $\msg(S)$ is called the {\it
embedding dimension }of $S$ and will be denoted by $\e(S).$

If $S$ is a numerical semigroup, then $\mathrm{F}(S)=\mbox{max}(\Z\backslash S)$,
 $\mathrm{g}(S)=\sharp(\N\backslash S)$, where $\sharp A$ denote the
cardinality of a set $A$, and  $\mbox{m}(S)=\mbox{min}(S\setminus \{0\})$.
 They are  three important invariants of
$S$ which we are known as the {\it Frobenius number}, the {\it genus} and  the {\it multiplicity} of $S$, respectively.

The study of numerical semigroups is a clasical topic and it has been motivated by the named Frobenius problem (see \cite{alfonsin}). It consists on finding formulas that calculate the Frobenius number and the genus of a numerical semigroup from its minimal system of generators. This problem was solved by Sylvester (see \cite{sylvester}) for numerical semigroups with embedding dimension two. Nowadays, the problem remains unsolved for  numerical semigroups with embedding dimension equal to or greater than or  three.

Let $\Delta$ be a numerical semigroup. An ideal of $\Delta$ is a nonempty subset $I$ of $\Delta$ such that $I+\Delta=\{a+b\mid a\in I \mbox{ and }b\in \Delta\}\subseteq I.$
For every ideal $I$ of $\Delta$, the set $I\cup \{0\}$ is also a numerical semigroup. This fact leads us to give the following definition: a numerical semigroup $S$ is an $\I(\Delta)$-{\it semigroup} if and only if $S\backslash \{0\}$ is an ideal of $\Delta$. We  denote by $\CJ(\Delta)$ the set $\{S \mid  S \text{ is an $\I(\Delta)$-semigroup} \}$, and if $\scrF$ is a family of numerical semigroups, then  $\CJ(\scrF)$ denotes the set $\bigcup_{\Delta \in \scrF}\CJ(\Delta).$

The main motivation of this work is that ordinary numerical and elementary numerical semigroups are those with a simpler structure, in that order. In that line we will prove that $\CJ(\N)$ is the set of ordinary numerical semigroup and that  $\CJ^2(\N)=\CJ(\CJ(\N))$ is equal to the set of elementary numerical semigroup.

Following the idea of the previous paragraph we give the definition of complexity of a numerical semigroup as follows: the {\it complexity } of a numerical semigroup $S$, denoted by $\C(S)$,  is the number $\min\{k\in \N\mid S \in \CJ^k(\N)\}.$

Let $S$ and $\Delta$ be numerical semigroups. We  say that $\Delta$ is {\it an ideal extension of } $S$ whenever $S\in \CJ(\Delta).$
An {\it i-chain of length n} connecting the numerical semigroups $S$ and $T$ is a chain of numerical semigroups $S_0\subseteq S_1 \subseteq S_2 \subseteq \dots \subseteq S_n$ such that $S_0=S,$ $S_n=T$ and $S_i \in \CJ(S_{i+1})$ for every $i\in \{0,\dots, n-1\}.$

The content of this work is organized as follows.
In the first section, we give some definitions and recall some results.
In Section 2, we present an algorithm that allows to build all the ideal extensions of a numerical semigroup $S.$ Section 3 is devoted to prove that $\C(S)$ is the minimun of the lenghts of the $i$-chain connecting $S$ and $\N.$ In Section 4, we show how we can build one of these chains with minimun lenght. As a consequence, we obtain that $\C(S)=\left\lfloor{\frac{\F(S)}{\m(S)}}\right\rfloor+1$ (where $\lfloor q \rfloor$ is defined as $\max\{z\in \Z\mid z\leq q\}$ for every $q\in \Q$).
In Section 5, we give an algorithm to compute all the numerical semigroups with fixed multiplicity and complexity. Finally, we prove that the cardinal of the set of numerical semigroups with multiplicity $m$ and complexity $c$ is less than or equal to the cardinality of the set of numerical semigroups with multiplicity $m$ and complexity equal to $c+1$.

\section{Ideal extensions of a numerical semigroup}
Let $S$ and $T$ be  numerical semigroups. We will say that $T$ is {\it an ideal extension} of $S$ if $S\backslash \{0\}$ is an ideal of $T$. The ideal extensions of semigroups were introduced  in \cite{clifford} and thenceforth  they have been extensively studied (see, for instant \cite{grillet}).

Following the notation introduced in \cite{JPAA}, we say that an  integer $x$ is a {\it pseudo-Frobenius number} of a numerical semigroup $S$, if $x \notin S$ and $x+s \in  S$ for
all $s \in  S\backslash \{0\}.$  We denote by $\PF(S)$ the set of pseudo-Frobenius numbers of
$S$. The cardinality of $\PF(S)$ is an important invariant of $S$ (see, for instance \cite{barucci}) that is called the {\it type} of $S$ and is denoted by $\t(S)$.

The following result is found in \cite[Corollary 2.23]{libro}.

\begin{proposition}\label{proposition1}
	If $S$ is a numerical semigroup such that $S\neq \N,$ then $\t(S)\leq \m(S)-1.$
\end{proposition}
Observe that if $S$ is a numerical semigroup such that $S\neq \N$ and $\{x,y\}\subseteq \PF(S),$ then $x+y \in S$ or $x+y \in \PF(S).$ Therefore, we can state the following result.

\begin{proposition}\label{proposition2}
If $S$ is a numerical semigroup and $S \neq \N,$ then $S \cup \PF(S)$ is also a numerical semigroup.
\end{proposition}

The following result indicates how the ideal extensions of a numerical semigroup are.

\begin{theorem}\label{theorem3}
Let $S$ and $\Delta$ be numerical semigroups. Then $\Delta$ is an ideal extension of $S$ if and only if $S\subseteq \Delta \subseteq S \cup \PF(S).$	
\end{theorem}
\begin{proof}
	{\it Necessity.} If $\Delta$ is an ideal extension of $S,$ then $S\subseteq \Delta$ and $(S\backslash \{0\})+\Delta \subseteq S\backslash \{0\}.$ Therefore, if $x\in \Delta \backslash S,$ then $\{x\}+(S\backslash \{0\})\subseteq S.$ Hence, $x\in \PF(S).$ Consequently, $\Delta \subseteq S\cup \PF(S).$
		
    	{\it Sufficiency.} If $S \subseteq \Delta \subseteq S\cup \PF(S),$ then $(S\backslash \{0\})+\Delta \subseteq S\backslash \{0\}.$ Thus, $S\backslash \{0\}$ is an ideal of $\Delta$ and so $\Delta$ is an ideal extension of $S.$
 \end{proof}

As an immediate consequence of the previous theorem, we have the following result.

\begin{corollary}\label{corollary4}
	Let $S$ be a numerical semigroup. Then the cardinality of the set $\{\Delta \mid \Delta \mbox{ is an ideal extension of }S\}$ is less than or equal to $2^{\t(S)}.$
	
\end{corollary}

Our purpose now is to present an algorithm to compute all the ideal extensions of a numerical semigroup $S$. In order to do that, we introduce the following concepts and results.

Let $S$ be a numerical semigroup and $n\in S\backslash \{0\}.$ We define, in honour of \cite{apery}, the {\it Apéry set of n in S}, as $\Ap(S,n)=\{s\in S\mid s-n\notin S\}.$

 The following result is in \cite[Lemma 2.4]{libro}.

 \begin{proposition}\label{proposition5}
 	Let $S$ be a numerical semigroup and  $n\in S \backslash \{0\}.$ Then $\Ap(S,n)$ has cardinality $n.$ Moreover,
 	$\Ap(S,n)=\{0=w(0),w(1), \dots, w(n-1)\}$, where $w(i)$ is the least
 	element of $S$ congruent with $i$ modulo $n$, for all $i\in
 	\{0,\dots, n-1\}.$ 	
 \end{proposition}

Let $S$ be a numerical semigroup.
We define over $\Z$  the following binary relation: $a\leq_S b$ if
$b-a \in S$. In \cite{libro} it is proved that $\leq_S $ is an order relation
(reflexive, antisymmetric and transitive).

The following result is found in \cite[Proposition 2.20]{libro}.
\begin{proposition}\label{proposition6} If $S$ is  a numerical semigroup and $n \in S\setminus
	\{0\},$ then
	$$
	\mathrm{PF}(S)=\{w-n\mid w \in {\rm Maximals}_{\leq_S} \Ap(S,n)\}.
	$$
\end{proposition}

We illustrate the content of the previous proposition with an example.

\begin{example}\label{example7} If $S=\{0,5,6,8,\rightarrow\}$ (the symbol $\rightarrow$ means that every integer greater than $8$ belongs to the set), then $\Ap(S,5)=\{0,6,8,9,12\}$ and ${\rm Maximals}_{\leq_S} \{0,6,8,\\
9,12\}=\{8,9,12\}.$ By applying Proposition \ref{proposition6}, we have $\PF(S)=\{3,4,7\}.$	
\end{example}

By Theorem \ref{theorem3}, we know that a numerical semigroup $\Delta$ is an ideal extension of a numerical semigroup $S$ if and only if there is $A\subseteq \PF(S)$ such that $\Delta=S \cup A.$
The following result is easy to prove and shows the property that a subset $A$ of $\PF(S)$ must verify  so that the set $S \cup A$ is a numerical semigroup.

\begin{proposition}\label{proposition8}Let $S$ be a numerical semigroup such that $S\neq \N$ and $A \subseteq \PF(S).$ Then the following conditions are equivalents:
	\begin{enumerate}
		\item[1)] $S\cup A$ is a numerical semigroup.
		\item[2)] If $\{a,b\}\subseteq A$ and $a+b\in \PF(S),$ then $a+b\in A.$
	\end{enumerate}
	
\end{proposition}
\begin{proof}
	{\it 1) implies 2).} If $\{a,b\}\subseteq A,$ then $a+b\in S\cup A.$ As $a+b\in \PF(S),$ then $a+b\notin S$ and so $a+b\in A.$
	
	{\it 2) implies  1).} The sum of two elements of $S$ belongs again to $S.$ As $A\subseteq \PF(S),$ then the sum of an element of $A$ and a nonzero element of $S$, belongs also to $S.$ Therefore, to prove that $S\cup A$ is a numerical semigroup, it will be enough to see that the sum of two elements belonging to $A$ is in $S\cup A.$ Indeed, if $\{a,b\}\subseteq A$ and $a+b\notin S,$ then by the previous comment to Proposition \ref{proposition2}, we know that $a+b\in \PF(S).$ Thus, $a+b\in A.$
\end{proof}

The above proposition leads to the following definition.

Let $S$ be a numerical semigroup. We say that a set $A$ is an $\ii(S)${\it -pertinent set} if it verifies the following conditions:
\begin{enumerate}
\item[1)]  $A\subseteq \PF(S)$ and
\item[2)] if $\{a,b\}\subseteq A$ and $a+b \in \PF(S),$ then $a+b\in A.$
\end{enumerate}

As an immediate consequence of Theorem \ref{theorem3} and Proposition \ref{proposition8}, we have the following result.
\begin{corollary}\label{corollary9} Let $S$ be a numerical semigroup such that $S\neq \N.$ Then the set formed by the ideal extensions of $S$ is $\{S\cup A\mid A \mbox{ is an }\ii(S)\mbox{-pertinent set}\}.$
\end{corollary}

We are already able to provide the  previously announced algorithm. \\

\begin{algorithm}[H]
	\BlankLine
	\KwData{A numerical semigroup $S\neq \N.$}
    \KwResult{$\{\Delta\mid \Delta \mbox{ is an ideal extension of }S\}.$}
    \BlankLine
    Compute $\PF(S)$\;
 	Compute $B=\{A \mid A \subseteq \PF(S) \mbox{ and } A \mbox{ is an }\ii(S)\mbox{-pertinent set}\}$\;
 	\Return $\{S\cup A\mid A\in B\}$\;
 	\caption{Computation of the set of ideal extensions of a numerical semigroup.}\label{algoritmo1}
\end{algorithm}

We finish this section illustrating  how the
previous algorithm works.
\begin{example} If $S=\{0,5,6,8,\rightarrow\},$ then by Example \ref{example7} we know that $\PF(S)=\{3,4,7\}.$ We calculate $\scrP(\{3,4,7\})=\{\emptyset, \{3\},\{4\},\{7\},\{3,4\},\{3,7\},\{4,7\},\\
\{3,4,7\}\}.$ A  simple check shows that $\{3,4\}$ is the only element of $\scrP(\{3,4,7\})$ which is not $\ii(S)$-pertinent. Algorithm 10 asserts that all the ideal extensions of $S$ are: $S\cup \emptyset$, $S\cup \{3\},$ $S\cup \{4\},$ $S\cup \{7\},$ $S\cup \{3,7\},$ $S\cup \{4,7\}$ and $S\cup \{3,4,7\}.$

The code below have is found in \href{https://github.com/D-marina/CommutativeMonoids/blob/master/ComplexityOfNS/complexityOfNS.ipynb}{complexityOfNS.ipynb}, is part of  \cite{commutative} and uses the library \cite{numericalsgps}.
\begin{verbatim}
gap> S:=NumericalSemigroup(5,6,8,9,10,11,12);;
gap> lDeltas:=idealExtensionsOfNS(S);;
gap> List(lDeltas,x->MinimalGeneratingSystemOfNumericalSemigroup(x));
\end{verbatim}
The output obtained is
\begin{verbatim}
[ [ 3, 5 ], [ 4, 5, 6 ], [ 5, 6, 7, 8, 9 ], [ 3, 4, 5 ],
[ 3, 5, 7 ], [ 4, 5, 6, 7 ] ]
\end{verbatim}
which is the list of system of generators of the proper ideal extensions of $S$.
\end{example}

\section{The complexity and the $\ii$-chains}
Our first objetive in this section is to prove that if $S$ is a numerical semigroup, there is $k\in \N$ such that $S\in \CJ^k(\N).$ Recall that $\CJ^0(\N)=\N,$ $\CJ^1(\N)=\CJ(\N)$ and $\CJ^{k+1}(\N)=\CJ(\CJ^{k}(\N))$ for all $k\in \N.$

\begin{proposition}\label{proposition12} Let $S$ be a numerical semigroup and $k\in \N.$ Then $S\in \CJ^k(\N)$ if and only if there exists an $\ii$-chain of lenght $k$ connecting $S$ with $\N.$	
\end{proposition}
\begin{proof}
	We proceed by induction on $k.$ For $k=0$ the result is trivial. Assume the result is true for $k-1$ and obtain the result for $k$.
	
	{\it Necessity.} If  $S\in \CJ^k(\N), $ then $S\in \CJ(\CJ^{k-1}(\N))$ and so there exists $S_1\in \CJ^{k-1}(\N)$ such that $S\subseteq S_1$ is an $\ii$-chain. By the induction hypothesis, there is an $\ii$-chain $S_1\subseteq S_2\subseteq \dots \subseteq S_k=\N$ of lenght $k-1$ connecting $S_1$ with $\N.$ It is clear that $S\subseteq S_1\subseteq S_2\subseteq \dots \subseteq S_k=\N$ is an $\ii$-chain of length $k$ connecting $S$ and $\N.$
	
	{\it Sufficiency.} If $S=S_0\subseteq S_1\subseteq \dots \subseteq S_k=\N $ is an $\ii$-chain of lenght $k$, then $S_1\subseteq \dots \subseteq S_k=\N $ is an $\ii$-chain of length $k-1.$ By the induction hypothesis $S_1\in \CJ^{k-1}(\N).$ Since $S\in \CJ(S_1),$ then $S\in \CJ(\CJ^{k-1}(\N))=\CJ^k(\N).$
\end{proof}

Let $\CL=\{S\mid S \mbox{ is a numerical semigroup}\}.$ We will say the map  $\theta:\CL\backslash \{\N\}\to \scrP(\N)$ is an {\it $\ii$-pertinent map} if $\theta(S)$ is a nonempty $\ii(S)$-pertinent set for all $S\in \CL \backslash \{\N\}.$

The following result shows us the large number of $\ii$-pertinent functions that exist.

\begin{proposition}\label{proposition13}
	The map $\theta:\CL\backslash \{\N\}\to \scrP(\N)$ defined by
	\begin{enumerate}
		\item $\theta(S)=\PF(S)$ is an $\ii$-pertinent map.
		\item $\theta(S)=\{\F(S)\}$ is an $\ii$-pertinent map.
		\item $\theta(S)=\{x\in\PF(S)\mid x\ge \frac{\F(S)}{2}\}$ is an $\ii$-pertinent map.
		\item $\theta(S)=\{x\in\PF(S)\mid x> \F(S)-\m(S)\}$ is an $\ii$-pertinent map.
		\item $\theta(S)=\{x\in\PF(S)\mid x> \F(S)-\g(S)\}$ is an $\ii$-pertinent map.
		\item $\theta(S)=\{\min \{x\in\PF(S)\mid x> \frac{\F(S)}{2}\}\}$ is an $\ii$-pertinent map.
		
	\end{enumerate}
	
\end{proposition}

If $\theta$ is  an $\ii$-pertinent map and $S$ is a numerical semigroup, then we can build a sequence of numerical semigroups as follows:
\begin{itemize}
	\item $S_0=S,$
	\item $S_{n+1}=\left\{\begin{array}{lr}
	S_n \cup \theta(S_n) &  \mbox{if }  S_n\neq \N\\
	\N & \mbox{otherwise.}
	\end{array}
	\right.$
\end{itemize}

\begin{theorem}\label{theorem14}
Let $S$ be a numerical semigroup, $\theta$ be  an $\ii$-pertinent map and $\{S_n\}_{n\in \N}$ be the sequence defined above. Then, there exists $\mu(\theta, S)\in \N$ such that $S=S_0\subsetneq S_1\subsetneq \dots \subsetneq S_{\mu(\theta, S)}=\N.$ Moreover, this chain is an $\ii$-chain.	
\end{theorem}
\begin{proof}
If $S_i\neq \N,$ then $\theta(S_i)\neq \emptyset$ and $\theta(S_i)\subseteq \N\backslash S_i.$ Therefore, $\g(S_{i+1})< \g(S_{i}).$ Thus there is $\mu(\theta, S)\in \N$ such that $S=S_0\subsetneq S_1\subsetneq \dots \subsetneq S_{\mu(\theta, S)}=\N.$ Moreover,
as a consequence of Corollary \ref{corollary9}, we have that  $S_i\in \CJ(S_{i+1})$ for all $i\in \{0,\dots, \mu(\theta, S)-1\}.$ Hence, $S=S_0\subsetneq S_1\subsetneq \dots \subsetneq S_{\mu(\theta, S)}=\N$ is an $\ii$-chain.
 \end{proof}
As a consequence of Proposition \ref{proposition12} and Theorem \ref{theorem14} we have the following result.

\begin{corollary}\label{corollary15}
	If $S$ is a numerical semigroup, then there exists $k\in \N$ such that $S\in \CJ^k(\N).$
	
\end{corollary}

The previous corollary allows us to give the following definition: {\it the complexity} of a numerical semigroup $S$, denoted by $\C(S)$, is the minimun of the set $\{k\in \N\mid S\in \CJ^k(\N)\}.$ Note that $\N$ is the unique numerical semigroup with complexity zero.

As an immediate consequence of Proposition \ref{proposition12}, we have the following result.

\begin{proposition}\label{proposition16} Let $S$ be a numerical semigroup. Then $\C(S)$ is the minimum of the lenghts of the $\ii$-chains connecting $S$ with $\N.$
\end{proposition}

The following result is an immediate consequence from previous proposition.

\begin{corollary}\label{corollary17}
	If $S$ is a numerical semigroup and $\theta$ is an $\ii$-pertinent map, then $\C(S)\leq \mu(\theta,S).$
\end{corollary}

We complete this section by showing the numerical semigroups with complexities one and two.
For this purpose, we need to recall the definition of ordinary numerical semigroups: a numerical semigroup $S$ is {\it ordinary} when $S$ is equal to $\{0,\m(S), \rightarrow \}.$

\begin{proposition}\label{proposition18}
	Let $S$ be a numerical semigroup. Then $\C(S)=1$ if and only if $S$ is an ordinary numerical semigroup and $S\neq \N.$	
\end{proposition}
\begin{proof}
	{\it Necessity.} If $\C(S)=1,$ then $S\subsetneq \N$
	is an $\ii$-chain and $S\in \CJ(\N)$. By Theorem \ref{theorem3}, $\PF(S)=\N\setminus S$, and thus $S=\{0,\m(S), \rightarrow \}$.
	
	{\it Sufficienty.} If $S$ is an ordinary semigroup and  $S \neq \N,$ then there is $m\in \N\backslash \{0,1\}$ such that $S=\{0,m, \rightarrow \}.$ It is clear that $\{0,m, \rightarrow \}\subsetneq \N$ is an $\ii$-chain. By applying now Proposition \ref{proposition16}, we have $\C(S)=1.$
\end{proof}
Following the notation introduced in \cite{blanco}, a numerical semigroup $S$ is {\it elementary } if $\F(S)<2\,\m(S).$
The following result is deduced from Lemma 2.1 of \cite{elementales}.

\begin{lemma}\label{lemma19}
	A numerical semigroup $S$ is elementary and not ordinary if and only if $S=\{0,m\}\cup A \cup \{2m,\rightarrow\}$ where $m\in \N\backslash \{0,1\}$ and $A\subsetneq \{m+1,m+2,\dots, 2m-1\}.$
	\end{lemma}
\begin{proposition}\label{proposition20}
	Let $S$ be a numerical semigroup. Then $\C(S)=2$ if and only if $S$ is an elementary and not ordinary semigroup.
\end{proposition}
\begin{proof}
	{\it Necessity.} If $\C(S)=2,$ then $S$ is not an ordinary semigroup since these semigroups have complexity $0$ or $1$. Moreover, by Proposition \ref{proposition16}, we know that there is an $\ii$-chain $S=S_0\subsetneq S_1 \subsetneq S_2=\N.$ Then we deduce that $\C(S_1)=1$ and by Proposition \ref{proposition18} we have $S_1=\{0,m, \rightarrow \}$ for some $m\in\N\backslash \{0,1\}.$ As $S\in \CJ(S_1),$ then $m\leq\m(S)$ and $\{\m(S)+m,\rightarrow\}\subseteq S.$ Thus, $\F(S)<\m(S)+m\leq 2\,\m(S).$ Hence $S$ is an elementary and not ordinary semigroup.
	
	{\it Sufficienty.} If $S$ is an elementary and not ordinary semigroup, then by Lemma \ref{lemma19}, we know that  $S=\{0,m\}\cup A \cup \{2m,\rightarrow\}$ with $m\in \N\backslash \{0,1\}$ and $A\subsetneq \{m+1,\dots, 2m-1\}.$ It is clear that $S \subsetneq \{0,m,\rightarrow\}\subsetneq \N$ is an $\ii$-chain. By Proposition \ref{proposition16}, $\C(S)=2.$
\end{proof}

\section{A formula for  the complexity}  	

Let $S$ be a numerical semigroup. By Proposition \ref{proposition18}, we know that $\C(S)=1$ if and only if $0 \cdot \m(S)<\F(S)<1\cdot \m(S).$ Also observe that from Proposition \ref{proposition20}, we deduce $\C(S)=2$ if and only if  $1 \cdot \m(S)<\F(S)<2\cdot \m(S).$ The following result generalizes these two properties.

\begin{theorem}\label{theorem21} Let $S$ be a numerical semigroup such that $S\neq \N.$ Then, $\C(S)=k$ if and only if $(k-1)  \m(S)<\F(S)<k \m(S).$	
\end{theorem}
\begin{proof}
	
We make this proof by induction on $k$. For $k\in\{1,2\}$ the result is true, so we assume $k\geq 3$.

{\it Necessity.}
If $\C(S)=k$, then, by the induction hypothesis, we have that $(k-1)\m(S)<\F(S)$. We also know that $S\backslash\{0\}$ is ideal of a numerical semigroup $\Delta$ such that $\C(\Delta)=k-1$. Using the induction hypothesis, we have $(k-2)\m(\Delta)<\F(\Delta)<(k-1)\m(\Delta)$. Since $S\backslash \{0\}$ is an ideal of $\Delta$, we have that  $\{\m(S)+(k-1)\m(\Delta),\to\}\subseteq S$ and $\m(\Delta)\leq \m(S)$. Hence $\{k\m(S),\to\}\subseteq S$, and thus $\F(S)< k\m(S)$.
	
{\it Sufficienty.}
If $(k-1)\m(S)<\F(S)$, then, by the induction hypothesis, $\C(S)\geq k$. Clearly $T=S\cup\{(k-1)\m(S),\to\}$ is a numerical semigroup. By the induction hypothesis, it fulfills that $\C(T)\leq k-1$. It is straightforward to prove that $S\backslash \{0\}$ is an ideal of $T$. Therefore $\C(S)\leq \C(T)+1=k$ and thus $\C(S)=k$.
\end{proof}

As an immediate consequence of previous theorem we have the following result.
\begin{corollary}\label{corollary22}
If $S$ is a numerical semigroup, then $\C(S)=\left\lfloor\frac{\F(S)}{\m(S)}\right\rfloor+1.$	
\end{corollary}
\begin{example}\label{example23}Let $S=\langle 5,7\rangle=\{0,5,7,10,12,14,15,17,19,20,21,22,24,25,\rightarrow\}.$ Then $\m(S)=5$ and $\F(S)=23.$ By applying Corollary \ref{corollary22}, we have that $\C(S)=\left\lfloor \frac{23}{5}\right\rfloor+1=5.$
	\end{example}

By Proposition \ref{proposition16}, we know that if $S$ is a numerical semigroup then there is an $\ii$-chain of length $\C(S)$ connecting $S$ and $\N.$ Our next aim in this section is to present an   $\ii$-chain with these conditions.

If $S$ is a numerical semigroup such that $S\neq \N,$ then we denote by $\gamma(S)=\{x\in\N\backslash S \mbox{ such that }\left\lfloor\frac{\F(S)}{\m(S)}\right\rfloor\m(S)\leq x \leq \F(S)\}.$ It is clear that the map $\gamma:\CL\backslash \{\N\}\to \scrP(\N)$ is an $\ii$-pertinent map.

\begin{proposition}\label{proposition24}
	Let $S$ be a numerical semigroup such that $S\neq \{0,\m(S),\rightarrow\}$ and $T=S\cup \gamma(S).$ The following properties are satisfied.
	\begin{enumerate}[(1)]
		\item $T$ is a numerical semigroup.
		\item $S\backslash \{0\}$ is an ideal of $T.$
		\item $\m(T)=\m(S).$
		\item $\C(T)=\C(S)-1.$
	\end{enumerate}
	
\end{proposition}
\begin{proof}
	(1) and (2). Since $\gamma$ is an $\ii$-pertinent map, $\gamma(S)$ is a noempty $\ii$-pertinent set. Therefore, $T=S\cup \gamma(S)$ is a numerical semigroup and $S\backslash \{0\}$ is an ideal of $T.$

	(3). Since $S\neq \{0,\m(S), \rightarrow\},$ then $\F(S)>\m(S)$ and so $\gamma(S)\subseteq \{\m(S),\rightarrow\}.$ Hence $\m(T)=\m(S).$
	
	(4). It is clear that $\left(\left\lfloor\frac{\F(S)}{\m(S)}\right\rfloor-1 \right)\m(S)<\F(T)< \left\lfloor\frac{\F(S)}{\m(S)}\right\rfloor\m(S).$ By applying that $\m(T)=\m(S)$, Corollary \ref{corollary22} and Theorem \ref{theorem21}, we have  $\C(T)=\C(S)-1.$	
	
\end{proof}

\begin{corollary}\label{corollary25}
	If $S$ is a numerical semigroup, then $\C(S)=\mu(\gamma,S).$
	
\end{corollary}

We illustrate the content of previous corollary with an example.

\begin{example}\label{example25}
	Let $S=\langle 5,7\rangle=\{0,5,7,10,12,14,15,17,19,20,21,22,24,25,\rightarrow\}.$ By Example \ref{example23}, we know that $\C(S)=5.$ Associated to the numerical semigroup $S$ and to the $\ii$-pertinent map (see Theorem \ref*{theorem14}) we have the $\ii$-chain $S=S_0\subsetneq S_1\subsetneq \dots \subsetneq S_{\mu(\gamma, S)}=\N,$ where $S_{n+1}=S_n\cup \gamma(S_n)$ for all $n\in \{0,\dots, \mu(\gamma,S)-1\}.$ Specifically, we have the chain
	\begin{multline*}
	S=S_0\subsetneq S_1=S\cup \{23\}\subsetneq S_2=S_1\cup \{16,18\}\subsetneq S_3=\\
	S_2\cup \{11,13\}\subsetneq S_4=S_3
	\cup \{6,8,9\}\subsetneq S_5=\N.
	\end{multline*}
	Corollary \ref{corollary25} tells us that this $\ii$-chain connects $S$ and $\N$ and it has minimum length.
{
This chain is also obtained with the function {\tt chainOfGamma} of \href{https://github.com/D-marina/CommutativeMonoids/blob/master/ComplexityOfNS/complexityOfNS.ipynb}{complexityOfNS.ipynb}:
\begin{verbatim}
gap> lGamma:=chainGamma(NumericalSemigroup(5,7));;
gap> List(lGamma,x->MinimalGeneratingSystemOfNumericalSemigroup(x));
\end{verbatim}
The output is the following list of minimal system of generators:
\begin{verbatim}
[ [ 5, 7, 23 ], [ 5, 7, 16, 18 ], [ 5, 7, 11, 13 ],
[ 5, 6, 7, 8, 9 ], [ 1 ] ]
\end{verbatim}
}
\end{example}

We end this section with some questions concerning the set of pseudo-Frobenius numbers.

\begin{remark}
    By Proposition \ref{proposition13}, the map $\theta:\CL\backslash \{\N\}\to \scrP(\N)$ defined by $\theta(S)=\PF(S)$ is an $\ii$-pertinent map.
    Since $\theta(S)$ is always the largest set among all the possible ones, we wonder if it is true that $\C(S)=\mu(\theta,S)$. Surprisingly, the answer is not. The numerical semigroup with the  smallest Frobenius number not satisfying that property is $S=\langle 4,6,9,11\rangle$ and the $\ii$-chain obtained is \begin{multline*}
        S=S_0\subsetneq S_1=\langle 2,5 \rangle=S\cup\{2,5,7\}\subsetneq S_2=\langle 2,3\rangle=S_1\cup\{3\}\subsetneq S_3=\N.
    \end{multline*} Note that $\C(S)=\lfloor \frac 7 4 \rfloor+1=2$, but the length of the $\ii$-chain is $3$.
    The next one is $\langle 5,7,9,11,13\rangle $ which has Frobenius number $8$. As we increase the Frobenius number, we obtain more examples. In addition to the obtained examples we can obtain new ones like for example  $\langle 4,6,9\rangle$ from our first example and  $\langle 5,7\rangle$ from our second example.

    If we examine the numerical semigroups verifying that $\C(S)=\mu(\theta,S)$, we see that the vast majority verify this property. Since apparently their proportion is smaller, an interesting question is to give a characterization of the family of numerical semigroups for which this property is not verified.

{
An example of a family that satisfies that adding $\PF(S)$ yields a chain that achieves complexity is that formed by the numerical semigroups of the form
$$S_k=\{0,m,2m,\dots,km\} \cup \{km+1,\to\}$$ with $k,m\in\N\backslash\{0\}$.
For this it is sufficient to take into account that $\PF(S_k)=\{(k-1)m+1,\dots,km-1\}$ and that $\lfloor \frac {km-1}{m} \rfloor +1=(k-1)+1=k$.
The $\ii$-chain obtained is $S_k\subsetneq S_{k-1}\subsetneq \dots \subsetneq S_1=\{0,m,\to\}\subsetneq \N$.
}

{
Consider now the numerical semigroups of the form $\{0,4,6,2\cdot 4, 3\cdot 4,\dots, k\cdot 4\}\cup\{4k+1,\to\}$ with $k\geq 2$. In the $\ii$-chain obtained with $\theta(S)=\PF(S)$, we find the semigroup $\langle 4,6,9,11\rangle=\{0,4,6,8,\to\}$. Hence, the length of this chain is larger than their complexity.
}
{
This can be generalized as follows.
Let $S=S_0$ be the numerical semigroup equal to $\{0,2k,3k,4k,\to\}$. The set $\PF(S)$ is $\{k,2k+1,\dots,4k-1\}$ and therefore $S_1=S\cup \PF(S)=\{0,k,2k,\to\}$. The set of pseudo-Frobenius numbers of $S_1$ is $\{k+1,\dots,2k-1\}$ and so $S_2=S_1\cup \PF(S_1)=\{0,5,\to\}$ which is an ordinary semigroup and the length of the $\ii$-chain obtained is $3$, but the complexity of $S$ is $\lfloor \frac {\F(S)}{m(S)}\rfloor+1=\lfloor \frac{4k-1}{2k} \rfloor+1=2$.
}
\end{remark}

\section{Numerical semigroups with a fixed multiplicity and complexity}

If $m\in \N\backslash \{0,1\}$, then we denote by
$$\CL_m=\{S\mid S \mbox{ is a numerical semigroup and } \m(S)=m\}.$$
For $S\in  \CL_m$ define the following sequence:

\begin{itemize}
	\item $S_0=S,$
	\item $S_{n+1}=\left\{\begin{array}{lr}
		S_n \cup \gamma(S_n), &  \mbox{if }  S_n\neq \{0,m,\rightarrow\}\\
		\{0,m,\rightarrow\}, & \mbox{otherwise.}
	\end{array}
	\right.$
\end{itemize}

The following result is obtained from Proposition \ref{proposition24} and Corollary \ref{corollary25}.
\begin{corollary}\label{corollary27}
	Let $m\in \N \backslash \{0,1\},$ $S\in \CL_m$ and let $\{S_n\}_{n\in \N}$ be the sequence defined above. Then,  $S=S_0\subsetneq S_1\subsetneq \dots \subsetneq S_{\C(S)-1}=\{0,m,\rightarrow\}.$
	In addition, we have the following:
	\begin{enumerate}
		\item $S_i\in \CL_m$ for all $i\in \{0,1,\dots,\C(S)-1\}.$
		\item $S_i\backslash \{0\}$ is an ideal of $S_{i+1}$  for all $i\in \{0,1,\dots,\C(S)-2\}.$
		\item $\C(S_i)=\C(S)-i$  for all $i\in \{0,1,\dots,\C(S)-1\}.$
	\end{enumerate}
	
\end{corollary}

A {\it graph} $G$ is a pair $(V,E)$, where $V$ is a nonempty set and $E$ is a subset of $\{(u,v)\in V \times V \mid u\neq v\}.$ The elements of $V$ and $E$ are called {\it vertices} and {\it edges} of $G$, respectively. A {\it path, of
	length $n$,} connecting the vertices $u$ and $v$ of $G$ is a
sequence of different edges of the form $(v_0,v_1),
(v_1,v_2),\ldots,(v_{n-1},v_n)$ such that $v_0=u$ and $v_n=v$.

We say that a graph $G$ is {\it a tree} if there exists a vertex $r$ (known as
{\it the root} of $G$) such that for any other vertex $v$ of $G$
there exists a unique path connecting $v$ and $r$. If $(u,v)$ is an edge of the tree $G$, then   we say that $u$ is a {\it child} of $v$.

If $m\in \{0,1\},$ then we define the graph $G(m)$ as follows: $\CL_m$ is its set of vertices and $(S,T)\in \CL_m\times \CL_m$ is an edge if $T=S\cup \gamma(S).$

As a consequence of Corollary \ref{corollary27}, we have the following result.
\begin{corollary}\label{corollary28}
	If $m\in \N\backslash \{0,1\},$ then $G(m)$ is a tree with root $\{0,m,\rightarrow\}.$ Moreover, $\{S\in \CL_m \mid \C(S)=k\}=\{S\in \CL_m \mid S \mbox{ is connected with }\{0,m,\rightarrow\} \mbox{through a path of length }k-1\}.$
	\end{corollary}

It is evident that a tree can be constructed recursively by starting from the root and connecting with an edge the vertices already constructed with their children. We are interested in characterizing how are the children of an arbitrary vertex of $\G(m)$.

\begin{proposition}\label{proposition29}
		Let $m\in \N\backslash \{0,1\}$ and $T\in \CL_m.$ Then, the set formed by all the children of $T$ in the tree $G(m)$ is $$\left\{T\backslash A \mid \emptyset \neq A \subseteq \left\{x\in \msg(T)\mid x>\left(\left\lfloor\frac{\F(T)}{\m(T)}\right\rfloor+1 \right)m\right\}\right\}.$$
	
\end{proposition}
\begin{proof}
	If $S$ is a child of $T$ in the tree $G(m),$ then $T=S\cup \gamma(S)$ and so $S=T\backslash \gamma(S).$ It is clear that $S$ is a numerical semigroup with complexity $\C(T)+1.$ Hence $\emptyset\neq \gamma(S)\subseteq \left\{x\in \msg(T)\mid x> \left( \left\lfloor\frac{\F(T)}{\m(T)}\right\rfloor+1\right)m\right\}.$
	
	Conversely, if $\emptyset\neq A\subseteq \{x\in \msg(T)\mid x> \left( \left\lfloor\frac{\F(T)}{\m(T)}\right\rfloor+1\right)m\},$ then $A\subseteq \left\{ \left(
	 \left\lfloor\frac{\F(T)}{\m(T)}\right\rfloor+1\right)m+1, \dots, \left(\left\lfloor\frac{\F(T)}{\m(T)}\right\rfloor+2\right)m-1
	\right\}$, since if $x\in \msg(T)$ then $x\leq \F(T)+m<\left( \left\lfloor\frac{\F(T)}{\m(T)}\right\rfloor+1\right)m+m=\left( \left\lfloor\frac{\F(T)}{m}\right\rfloor+2\right)m.$ Thus, $S=T\backslash A$ is a numerical semigroup and $\gamma(S)=A.$ Consequently, $S\in \CL_m$ and $S\cup \gamma(S)=T.$ Therefore, $S$ is a child of $T.$
\end{proof}	
\begin{example}\label{example30}
    Next we can see the tree $G(2)$.\\
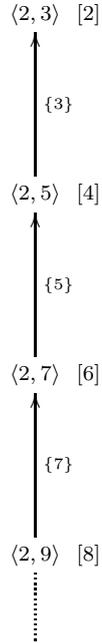
\begin{center}
{\footnotesize
    \xymatrixcolsep{0mm}
    \xymatrix{
        & & & & &  &&& \langle 2,3 \rangle  & [2]   & & &
        \\
        \\
        & & & & &   &&& \langle 2,5 \rangle \ar[uu]_{\{3\}} & [4] &  &  &
        \\
        \\
        & & & & &    &&&\langle 2,7 \rangle \ar[uu]_{\{5\}} & [6] &  &  &
        \\
        \\
            & & & & &   &&& \langle 2,9 \rangle  \ar[uu]_{\{7\}}\ar@{.}[d] & [8] &  &  &
            \\
            & & & & &   &&&  &  &  &  &
           }
}
\end{center}
    The number $[k]$ indicates $\left(\left \lfloor \frac {\F(T)}  {2}\right \rfloor+1\right)2$ and $\begin{array}{l}P\\    \Big \uparrow {\{x\}}\\ Q \end{array}$ means that $Q=P\backslash \{x\}.$
    We see that the vertices of $G(2)$ are of the form $\langle 2,2k+1\rangle $ where $k$ is the complexity of the semigroup.
\end{example}

\begin{example}\label{example31}
We now use Proposition \ref{proposition29} to obtain the tree $G(3)$.\\
{\footnotesize 
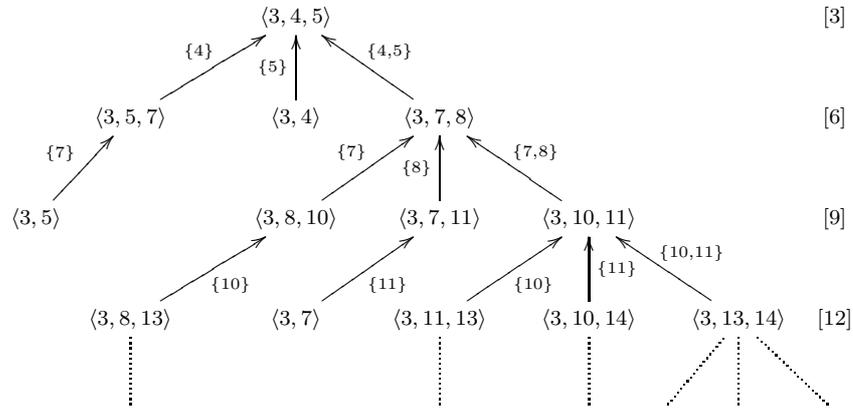
\xymatrix@C=0.5em{
    & & & & & & & \langle 3,4,5\rangle  & & & & &  &  & \,[3]  \\
    && & & \langle 3,5,7 \rangle \ar[urrr]^{\{4\}} & & & \langle 3,4\rangle \ar[u]^{\{5\}} & & \langle 3,7,8\rangle \ar[ull]_{\{4,5\}} & & & & & \,[6]\\
    && & \langle 3,5 \rangle \ar[ur]^{\{7\}}& & & & \langle 3,8,10\rangle \ar[urr]^{\{7\}} &&\langle 3,7,11\rangle \ar[u]^{\{8\}} & & \langle 3,10,11\rangle \ar[ull]_{\{7,8\}} & & & \,[9]\\
    & & & & \langle 3,8,13 \rangle \ar[urrr]_{\{10\}}\ar@{.}[d] &  &&\langle 3,7 \rangle \ar[urr]_{\{11\}} && \langle 3,11,13\rangle \ar[urr]_{\{10\}}\ar@{.}[d] && \langle 3,10,14\rangle \ar[u]_{\{11\}}\ar@{.}[d]&& \langle 3,13,14\rangle \ar[ull]_{\{10,11\}}\ar@{.}[rd]\ar@{.}[d]\ar@{.}[ld]
    & \,[12]\\
    &&  &    &  & &   & &  & & & & & &}
}
\end{example}

Our next aim is to present an algorithm to compute all the numerical semigroups with a given multiplicity and complexity. For this reason we introduce some concepts and results.

If  $G=(V,E)$ is a tree and $v$ is a vertex  of $G$, then the {\it
	depth} of $v$, denoted by $\d(v)$, is the length  of the only
path connecting   $v$ with the root.
The following result is an immediate consequence of Corollary \ref{corollary28}.

\begin{proposition}\label{proposition32}
	If $m\in \N\backslash \{0,1\}$ and $S$ is a vertex of $G(m),$ then $G(S)=\d(S)+1$.
	
\end{proposition}

If $G=(V,E)$ is a tree, then  we denote
by $\NN(G,n)=\{x\in V\mid d(x)=n\}.$

\begin{example}\label{example33} From Example \ref{example31}, we easily deduce $\NN(G(3),2)=\{\langle 3,5 \rangle, \langle 3,8,10 \rangle,\\ \langle 3,7,11 \rangle, \langle 3,10,11 \rangle\}.$ Therefore, by applying Proposition \ref{proposition32}, we have $\{S\in \CL_3\mid \C(S)=3\}=\{\langle 3,5 \rangle, \langle 3,8,10 \rangle, \langle 3,7,11 \rangle, \langle 3,10,11 \rangle\}.$
	
\end{example}
The proof of the following result is straightforward.
\begin{proposition}\label{proposition34}
If $G=(V,E)$ is a tree and $r$ is its root, then $\NN(G,0)=\{r\}$ and $\NN(G,n+1)=\{v\in V\mid v \mbox{ is a child of a vertex from }\NN(G,n)\}$ for all $n\in \N.$
\end{proposition}

The algorithm for the calculation of the above sets is as follows.

\begin{algorithm}[h]
	\BlankLine
	\KwData{An positive integer $c$ and $m\in \N\backslash \{0,1\}.$}
    \KwResult{$\{S\in \CL_m\mid \C(S)=c\}.$}
    \BlankLine
	$A:=\{\{0,m,\rightarrow\}\}$\; $i:=1$\;
	\If{$i=c$}{\Return $A$\;}\label{marker}
	Compute $B:=\{S\in \CL_m\mid S \mbox{ is a child of an element of }A\}$\;
	$A:=B$\;
	$i:=i+1$\;
	{\bf Goto} line 3\;
	\caption{Computation of the set of numerical semigroups with multiplicity $m$ and complexity $c$.}
\end{algorithm}

We proceed to illustrate how the
previous algorithm works with an example.
\begin{example}\label{example36}
    We proceed now to compute the set $\{S\in \CL_3\mid \C(S)=4\}$ using the previous algorithm.
	\begin{itemize}
		\item $A=\{\{0,3,\rightarrow\}\},$ $i=1.$
		\item $A=\{\langle 3,5,7\rangle,\langle 3,4\rangle, \langle 3,7,8\rangle\},$ $i=2.$
		\item $A=\{\langle 3,5\rangle,\langle 3,8,10\rangle, \langle 3,7,11\rangle, \langle 3,10,11\rangle\},$ $i=3.$
		\item $A=\{\langle 3,8,13\rangle,\langle 3,7\rangle, \langle 3,11,13\rangle, \langle 3,10,14\rangle, \langle 3,13,14\rangle\},$ $i=4.$
		\end{itemize}
	Thus, $\{S\in \CL_3\mid \C(S)=4\}=\{\langle 3,8,13\rangle,\langle 3,7\rangle, \langle 3,11,13\rangle, \langle 3,10,14\rangle, \langle 3,13,14\rangle\}.$
	
This list is obtained with the function {\tt NSWithMultiplicityAndComplexity} of \href{https://github.com/D-marina/CommutativeMonoids/blob/master/ComplexityOfNS/complexityOfNS.ipynb}{complexityOfNS.ipynb}:
\begin{verbatim}
gap> l34:=NSWithMultiplicityAndComplexity(3,4);;
gap> List(l34,x->MinimalGeneratingSystemOfNumericalSemigroup(x));
\end{verbatim}
The result obtained is:
\begin{verbatim}
[ [ 3, 8, 13 ], [ 3, 7 ], [ 3, 11, 13 ], [ 3, 10, 14 ],
[ 3, 13, 14 ] ]
\end{verbatim}
\end{example}
	
	
The following result easily follows.

 \begin{lemma}\label{lemmaA}
 If $S$ is a numerical semigroup, then $T=(\{\m(S)\}+S)\cup \{0\}$ is again a numerical semigroup. In addition, $\m(T)=\m(S)$, $\F(T)=\F(S)+\m(S)$ and $\C(T)=\C(S)+1.$
 \end{lemma}

 The following proposition can be  deduced from  previous lemma.

 \begin{proposition}\label{propositionB}
 The map $f:\{S\in \CL_m\mid \C(S)=c\}\to \{S\in \CL_m\mid \C(S)=c+1\}$ defined by $f(S)=(\{\m(S)\}+S)\cup \{0\}$, is injective.
 \end{proposition}

 As an immediate consequence from above proposition, we have the following result.

 \begin{corollary}\label{corollaryC}
 If $c\in \N$ and $m\in \N\backslash \{0\}$, then the cardinality of $\{S \in \CL_m\mid \C(S)=c\}$ is less than or equal to the cardinality of $\{S\in \CL_m\mid \C(S)=c+1\}.$
 \end{corollary}



\begin{thebibliography}{12}
	
 \bibitem{apery}
 \textsc{R. Ap\'{e}ry},
 \newblock \emph{Sur les branches superlin\'{e}aires des courbes alg\'{e}briques,}
 \newblock   C.R. Acad. Sci. Paris 222 (1946), 1198--2000.




\bibitem{barucci}
\textsc{V. Barucci, D. E. Dobbs and M. Fontana},
\newblock \emph{Maximality Properties in Numerical Semigroups and Applications to One-Dimensional Analitycally Irreducible Local Domains,}
\newblock   Memoirs Amer. Math. Soc. 598 (1997).


\bibitem{blanco}
\textsc{V. Blanco and J. C. Rosales},
\newblock \emph{The set  of  numerical semigroups of a given genus,}
\newblock Forum Math. 85 (2012), 255--267.

\bibitem{numericalsgps}
\textsc{M. Delgado, P. A. Garcia-Sanchez and J. Morais}, \newblock \emph{NumericalSgps, a package for numerical semigroups},
\newblock Version 1.2.0.
\href{https://gap-packages.github.io/numericalsgps}{https://gap-packages.github.io/numericalsgps}, Apr 2019. Refereed GAP package.

\bibitem{clifford}
\textsc{A. H. Clifford},
\newblock \emph{Extensions of semigroups,}
\newblock  Trans. Amer. Math. Soc. 68 (1950), 165--173.



\bibitem{commutative}
 \textsc{J. I. Garc\'{\i}a-Garc\'{\i}a, D. Mar\'{\i}n-Arag\'{o}n, A. Sánchez-R.-Navarro, and A. Vigneron-Tenorio},
\newblock {\em CommutativeMonoids, computations in finitely generated commutative monoids}.
\newblock Available at \url{https://github.com/D-marina/CommutativeMonoids}.


\bibitem{grillet}
\textsc{P. A. Grillet},
\newblock \emph{Semigroups. An Introduction to the structure theory, }
\newblock   Marcel Dekker, New York, 1995.


\bibitem{alfonsin}
\textsc{J. L. Ramírez Alfonsín},
\newblock \emph{The Diophantine Frobenius Problem,}
\newblock Oxford University Press, 2005.



\bibitem{JPAA}
\textsc{J. C. Rosales, and  M. B. Branco},
\newblock \emph{Numerical Semigroups that can be expressed as an intersection of symmetric numerical semigroups,}
\newblock   J. Pure Appl. Algebra  171 (2002), 303--314.

\bibitem{elementales}
\textsc{J. C. Rosales,  M. B. Branco},
\newblock \emph{On the enumeration of the set of elementary numerical semigroups with fixed multiplicity, Frobenius number or genus,}
\newblock  Kragujevac J.  Math.  46 (2022), 433--442.
%


\bibitem{libro}
\textsc{J. C. Rosales and P. A. Garc\'ia-S\'anchez},
\newblock \emph{Numerical Semigroups,}
\newblock   Developments in Mathematics, Vol. 20, Springer, New York, 2009.


\bibitem{sylvester}
\textsc{J. J. Sylvester},
\newblock \emph{Problem 7382,}
\newblock Educat. Times J. College Preceptors New Ser. 36 (1883), 177.


\end{thebibliography}
\end{document}